\documentclass[11pt]{amsart}
\usepackage{graphicx,amssymb,amsmath, latexsym,cite, amscd,stmaryrd}
\usepackage[foot]{amsaddr}  
\usepackage{amsthm,amsrefs} 
\usepackage{enumerate,lscape} 
\usepackage{float,xcolor}
\usepackage{fancyhdr, lastpage}  
\usepackage{libertine}    
\usepackage{url}  
\usepackage{verbatim,lineno}    
      
\usepackage{tikz-cd}   
\usepackage{mathtools}  
\usepackage{amsfonts}

\newtheoremstyle{mythm}                   
{6pt}
{6pt}
{\it}
{}
{\bf}
{.}
{.5em}
{}

\newtheoremstyle{mydef}                   
{6pt}
{6pt}
{}
{}
{\bf}
{.}
{.5em}
{}

\newtheoremstyle{myrem}                   
{6pt}
{6pt}
{}
{}
{\bf}
{.}
{.5em}
{}

\theoremstyle{mythm}
\newtheorem{theorem}{Theorem}[section]
\newtheorem{proposition}[theorem]{Proposition}
\newtheorem{lemma}[theorem]{Lemma}
\newtheorem{corollary}[theorem]{Corollary}

\theoremstyle{mydef}

\newtheorem{example}[theorem]{Example}

\theoremstyle{myrem}
\newtheorem{remark}[theorem]{Remark}
\numberwithin{equation}{section}

\newcounter{ithmcount}
\newenvironment{iprf}{\begin{list}{{\rm
	\alph{ithmcount})}}{\usecounter{ithmcount}\labelwidth-5pt
      \leftmargin0pt \topsep3pt \itemsep1pt \parsep2pt}}{\qedhere\end{list}}

\newenvironment{ithm}{\begin{list}{{\rm \alph{ithmcount})}}{\usecounter{ithmcount}\labelwidth18pt
      \leftmargin18pt \topsep3pt \itemsep1pt \parsep2pt}}{\end{list}}

\allowdisplaybreaks

\reversemarginpar
\subjclass[2000]{}

\DeclareMathOperator{\im}{\rm im}
\newcommand{\Aut}{{\rm Aut}}

\newcommand{\GL}{{\rm GL}}

\newcommand{\Z}{\mathbb{Z}}
\newcommand{\SL}{{\rm SL}}

\renewcommand{\leq}{\leqslant}

\renewcommand{\geq}{\geqslant}

\newcommand{\ic}[1]{#1^\ast}
\newcommand{\icb}[1]{(#1)^\ast}

\begin{document}

\vspace*{-1.5cm}

\title[Group functor]{On two group functors extending Schur multipliers}
\author[H. Dietrich]{Heiko Dietrich}
\address{School of Mathematics, Monash University,
  Clayton VIC 3800, Australia}
\email{heiko.dietrich@monash.edu}
\author[P. Moravec]{Primo\v{z} Moravec}
\address{Faculty of Mathematics and Physics, University of Ljubljana, Slovenia}
\email{primoz.moravec@fmf.uni-lj.si}

\date{\today}

\begin{abstract}
 Liedtke (2008) has introduced group functors $K$ and $\tilde K$, which are used in the context of describing certain invariants for complex algebraic surfaces. He proved that these functors are connected to the theory of central extensions and Schur multipliers. In this work we relate $K$ and $\tilde K$ to a group functor $\tau$ arising in the construction of the non-abelian exterior square of a group. In contrast to $\tilde K$, there exist efficient algorithms for constructing $\tau$, especially for polycyclic groups. Supported by computations with the computer algebra system GAP, we investigate when $K(G,3)$ is a quotient of $\tau(G)$, and when $\tau(G)$ and $\tilde K(G,3)$ are isomorphic.
\end{abstract}

\thanks{This research was supported through the programme ``Research in Pairs'' by the Mathematisches Forschungsinstitut Oberwolfach in 2018;
the authors thank the MFO for the great hospitality.
Moravec acknowledges the financial support from the Slovenian Research Agency (research
core funding No.\ P1-0222, and projects No.\ J1-8132, J1-7256 and N1-0061). Dietrich's research is supported by an Australian Research Council grant, identifier DP190100317. Both authors thank the referees for the thorough reading and  for providing many details that made some of our results stronger.}
\maketitle 

\vspace{-0.6cm}

\section{Introduction}

\noindent
In the study of complex algebraic surfaces it is of interest to find strong invariants which are not too complicated to be useful. Towards this aim, Liedtke \cite{Lie08} introduced group theoretical functors $K$ and $\tilde{K}$ that are related to the fundamental groups of the associated Galois closures. More precisely,
let $X$ be a smooth projective surface, fix a generic projection $f\colon X\to \mathbb{P}^2$ of degree $n$, and let $f_{\rm gal}\colon X_{\rm gal}\to \mathbb{P}^2$ be its Galois closure. Let $\mathbb{A}^2$ be the complement of a fixed generic line in $\mathbb{P}^2$, and set $X^{\rm aff}=f^{-1}(\mathbb{A}^2)$ and $X_{\rm gal}^{\rm aff}=f_{\rm gal}^{-1}(\mathbb{A}^2)$. It is proved in \cite[Theorems 5.1 \& 5.2]{Lie08} that $\pi_1(X^{\rm aff}_{\rm gal})$  has images isomorphic to $\tilde{K}(\pi_1(X^{\rm aff}), n)$ and to $K(\pi_1(X^{\rm aff}), n)$. It is the constructions of $K(-, n)$ and $\tilde{K}(-,n)$ that are central to Liedtke's investigation in \cite{Lie08, Lie10}. As pointed out in these papers, it is important to have a better understanding of $\tilde{K}$ in order to describe the above mentioned fundamental groups.

The aim of this work is to extend the group theoretical analysis of the functors $\tilde{K}$ and $K$, and to relate these to a functor $\tau$ associated with Brown and Loday's construction of the non-abelian tensor square of a group \cite{BL87}. The latter has applications in topology and K-theory, and can efficiently be computed for several classes of groups, such as polycyclic groups.
 
In Section \ref{sec:prelim}, we set the notations and give the definitions of $K(G,n)$, $\tilde{K}(G,n)$, and  $\tau(G)$.
In Section \ref{sec:explicit}, we elaborate on these and provide explicit descriptions that enable efficient computations for polycyclic groups.
In Section \ref{sec:tau}, we introduce the concept of an AI-automorphism and show that the existence of such an automorphism for a group $G$ yields a central extension
\[\begin{tikzcd}
    1\arrow{r} & H_2(G,\Z) \arrow{r} & \tau(G) \arrow{r} & K(G,3) \arrow{r} & 1,
\end{tikzcd}\]
similar to the one proved in \cite[Theorem 2.2]{Lie08}:
 \[\begin{tikzcd}
    1\arrow{r} & H_2(G,\Z) \arrow{r} & \tilde{K}(G,3) \arrow{r} & K(G,3) \arrow{r} & 1.
\end{tikzcd}\]
It is therefore natural to ask when $\tau(G)$ and $\tilde{K}(G,3)$ are isomorphic. In Section \ref{sec:isom}, we explore this question for several classes of groups. For example, we show that if $G$ is a finite group and a Schur cover $H/M=G$ admits an AI-automorphism which acts as inversion on $M$, then $\tau(G)\cong \tilde{K}(G,3)$.

In Section \ref{sec:bog}, we show that $K(G,3)$ and $\tilde{K}(G,3)$ are closely related to the unramified Brauer group of the field of $G$-fixed points in a complex function field.  This group is also known as the Bogomolov multiplier $B_0(G)$, and has various applications in algebraic geometry, in particular, to Noether's Problem. In Section \ref{sec:comput} we comment on our computational experiments with the system GAP \cite{gap}.

\section{Definitions and preliminary results}
\label{sec:prelim}

\noindent 
Unless stated otherwise, all groups are finite and written multiplicatively. For a group $G$ and integer $n>0$ we denote by $G^n$ the direct product of $n$ copies of $G$. We write $C_n$ for the cyclic group of size $n$. The commutator subgroup $G'$ is the subgroup of $G$ generated by all commutators $[g,h]=g^{-1}h^{-1}gh=g^{-1}g^h$ with $g,h\in G$. A free presentation for $G$ is a free group $F$ with normal subgroup $N\unlhd F$ such that $G\cong F/N$; since $G$ is assumed to be finite, we assume that $F$ is finitely generated. A polycyclic presentation ${\rm pc}\langle g_1,\ldots,g_n\mid r_1,\ldots,r_m\rangle$ for $G$ is a group presentation with abstract generators $g_1,\ldots, g_n$ and  relations $r_1,\ldots,r_m$ that are power or conjugate relations, with the convention that trivial conjugate relations are omitted; see \cite[Section 2.1]{EN} for details. For example, ${\rm pc}\langle g_1,g_2\mid g_1^2,g_1^2\rangle$ describes the Klein 4-group $\langle g_1,g_2\mid g_1^2, g_2^2, g_2^{g_1}=g_2\rangle$. A group extension of $A$ by $B$ is written $G=B.A$, meaning that $A\unlhd G$ with quotient $G/A=B$.

\subsection{Liedtke's constructions}
For a group $G$ and integer $n\geq 2$, the group $K(G,n)$ is the kernel of the map $G^n\to G/G'$ that sends an $n$-tuple $(g_1,\ldots,g_n)$ to the product of its components modulo the commutator subgroups, that is,
\[K(G,n)=\{(g_1,\ldots,g_n)\in G^n: g_1\cdots g_n\in G'\}.\]Note that every permutation of the $n$ factors in $G^n$ defines an automorphism of $K(G,n)$, that is, we have ${\rm Sym}_n\leq \Aut(K(G,n))$. To define the group $\tilde K(G,n)$, choose a free presentation $G=F/N$ for $G$, and set
\[\tilde K(G,n)=K(F,n)/K(N,n)^{F^n},\]
where $K(N,n)^{F^n}$ is the normal closure of $K(N,n)$ in $F^n$; if $n\geq 3$, then this is simply the normal closure of $K(N,n)$ in $K(F,n)$, see \cite[p.\ 248]{Lie08}. It is shown in \cite[Theorem 2.2]{Lie08} that the  definition of $\tilde K(G,n)$ does not depend on the choice of presentation for $G$.

\subsection{Non-abelian exterior square}
Let $G$ and $\ic{G}$ be groups, with isomorphism $G\to\ic{G}$, $g\mapsto \ic{g}$; we continue to use ``$\ast$'' to denote elements and subsets of $\ic{G}$.   Let $G\star \ic{G}$ be the free product of $G$ and $\ic{G}$, and, following \cite{Roc91}, define $\nu(G)$ as a quotient group of $G\star \ic{G}$ via
\[\nu(G)=(G\star \ic{G})/ \langle \{[x,\ic{y}]^z[x^z,\icb{y^z}]^{-1}, [x,\ic{y}]^{(\ic{z})}[x^z,\icb{y^z}]^{-1}: x,y,z\in G\}\rangle^{G\star\ic{G}}.\]
To simplify notation, we identify elements in $\nu(G)$ with elements in $G\star\ic{G}$, keeping in mind that further relations hold in $\nu(G)$. If we want to emphasise the parent group, then we sometimes use subscripts at generated groups $\langle -\rangle_A$ or at commutators $[-,-]_A$ to indicate that the corresponding structures are to be considered in the group $A$. For example, if $g\in G$ and $\ic{g}\in\ic{G}$, then $[g,\ic{g}]_{\nu(G)}$ denotes their commutator in $\nu(G)$, not in $G\ast \ic{G}$.  With this convention, consider $\nabla(G)=\langle [x,\ic{x}]_{\nu(G)}: x\in G\rangle$ as a subgroup of $\nu(G)$, and define
\[\tau(G)=\nu(G)/\nabla(G).\]
Note that the homomorphism $G\star\ic{G}\to G\times G$, $g_1\ic{h_1}g_2\ic{h_2}\ldots g_k\ic{h_k}\mapsto (g_1\cdots g_k,h_1\cdots h_k)$, maps commutators $[x,\ic{y}]$ to $1$, hence it induces short exact sequences
\[\begin{tikzcd}
    1\arrow{r} & G\otimes G \arrow{r} & \nu(G) \arrow{r}{c_\nu} & G\times G \arrow{r} & 1 \\[-3ex]
    1\arrow{r} & G\wedge G \arrow{r} & \tau(G) \arrow{r}{c_\tau} & G\times G \arrow{r} & 1
\end{tikzcd}
\]
where the kernels $G\otimes G$ and $G\wedge G$ are called the {\em non-abelian tensor square} and the {\em non-abelian exterior square} of $G$, respectively. In will be shown in Lemma \ref{lem1} below that this coincides with the definitions given in \cite{BL87}. We conclude with a lemma that is used later.

\begin{lemma}\label{lem:kerRoc}
Let $H\to G$ be a surjective group homomorphism with kernel $M$. Then there are induced epimorphisms $\beta\colon \nu(H)\to \nu(G)$ and $\gamma\colon \tau(H)\to \tau(G)$ whose kernels are
\[\langle M,\ic{M}\rangle_{\nu(H)} [M,\ic{H}]_{\nu(H)}[H,\ic{M}]_{\nu(H)}\quad\text{and}\quad\langle M,\ic{M}\rangle_{\tau(H)} [M,\ic{H}]_{\tau(H)}[H,\ic{M}]_{\tau(H)}.\]
\end{lemma}
\begin{proof}
For $\beta$ this is \cite[Proposition 2.5]{Roc91}. Since $\beta$ maps $\nabla(H)$ to $\nabla(G)$, this  induces  $\gamma$. Note that $\ker\gamma=\{x\cdot \nabla(H): x\in\beta^{-1}(\nabla(G))\}$, and $\beta^{-1}(\nabla(G))=(\ker\beta) \cdot \nabla(H)$; the claim follows.
\end{proof}


\subsection{Schur multiplier}\label{secSM}
 
We recall some facts about the Schur multiplier of a finite group and refer to \cite{Kap} for more details, in particular, Proposition 2.1.1 and Theorems 2.1.4, 2.4.6, 2.5.1, 2.6.7, and 2.7.3.  A \emph{Schur cover} of $G$ is a group $H$ such that $H/M\cong G$ for some $M\leq H'\cap Z(H)$ isomorphic to the {\em Schur multiplier} \[M(G)=H^2(G,\mathbb{C}^\times).\] Note that $G'\cong H'/M$ since $M\leq H'$. Schur (1904-07) has shown that $M(G)$ is finite and if $F/N=G$ is a free presentation of $G$ with $F$ a free group of finite rank $r$, then $M(G)\cong(F'\cap N)/[F,N]$; the latter is  known as Hopf's formula. Every Schur cover $H$ of $G$ is isomorphic to $F/S$ for some normal subgroup $S\unlhd F$ that defines a complement $S/[F,N]$ to $(F'\cap N)/[F,N]\cong M(G)$ in $N/[F,N]$; in particular, $S/[F,N]$ is free abelian of  rank $r$ and $(F'\cap N)/[F,N]$ is the torsion subgroup of $N/[F,N]$. The isomorphism type of a Schur cover is in general not uniquely determined. However, Schur proved that the isomorphism type of $H'$ depends only on $G$, and not on the chosen cover $H$. Miller (1952) has shown that  \[M(G)\cong H_2(G,\Z).\]   

We will see in Remark \ref{remId} below that we can identify $[G,\ic{G}]_{\tau(G)}=G\wedge G$ via $[g,\ic{h}]\mapsto g\wedge h$. This identification allows us to define the surjective commutator map \[\kappa\colon G\wedge G\to G',\quad g\wedge h\mapsto [g,h],\]which, according to \cite[Corollary 2]{BJR},  can be lifted to an isomorphism \[G\wedge G\to H',\quad g\wedge h\to [g',h'],\] where $g',h'\in H$ are lifts of $g,h\in G$. Since $G'=H'/M(G)$, this yields  an exact sequence
\[\begin{tikzcd}
    1\arrow{r} & M(G)\arrow{r} & G\wedge G \arrow{r}{\kappa} & G'\arrow{r} & 1
\end{tikzcd}
\]with $\ker\kappa\cong M(G)$ central in $G\wedge G$. This shows that if $G$ is abelian, then $G\wedge G\cong M(G)\cong H'$, and a Schur cover of $G$ is abelian if and only if $G$ is cyclic if and only if $M(G)=1$.

\section{Explicit description}\label{secExp}
\label{sec:explicit}
As a first step towards investigating the relation between $\tau(G)$ and $\tilde K(G,3)$ we provide a more concrete description of these groups.

\subsection{An explicit description of $\tau$}
The next lemma summarises some facts about $\tau(G)$ and $\nu(G)$. 

\begin{lemma}\label{lem1}
Every $w\in\nu(G)$ can be written uniquely as $w=g\ic{h}w'$ for some  $w'\in [G,\ic{G}]_{\nu(G)}$ and $g,h\in G$; the analogous statement holds in $\tau(G)$. Moreover, we have  \[\ker c_\nu=[G,\ic{G}]_{\nu(G)}\cong G\otimes G\quad\text{and}\quad \ker c_\tau=[G,\ic{G}]_{\tau(G)}\cong G\wedge G.\]
\end{lemma}
\begin{proof}
  Let $g=g_1\ic{h_1}\cdots g_n\ic{h_n}\in\nu(G)$. The identities 
  \begin{eqnarray}\label{eqIds}\ic{h}g=g\ic{h}[\ic{h},g],\quad [\ic{h},g]k=k[\icb{h^k},g^k],\quad\text{and}\quad [\ic{h},g]\ic{k}=\ic{k}[\icb{h^k},g^k]
  \end{eqnarray}
  can be used to rewrite  $g=g_1\ic{h_1}\cdots g_n\ic{h_n}=(g_1\cdots g_n)\icb{h_1\cdots h_n}w$ with  $w\in [G,\ic{G}]_{\nu(G)}$. Recall that $c_\nu$ maps $[G,\ic{G}]_{\nu(G)}$ to 1, hence  $c_\nu(g)=(g_1\cdots g_n,h_1\cdots h_n)$, which proves  $\ker c_\nu= [G,\ic{G}]_{\nu(G)}$. The uniqueness of the expression of $g$ follows from the exact sequence associated with $c_\nu$. The argument for $\tau(G)$ and $c_\tau$ is exactly the same. Recall that above we have defined $G\otimes G=\ker c_\nu$ and $G\wedge G=\ker c_\tau$; it is shown in  \cite[Proposition 2.6]{Roc91} that the 'non-abelian tensor square' of  \cite{BL87} is isomorphic to $[G,\ic{G}]_{\nu(G)}$ via $[g,\ic{h}]\mapsto g\otimes h$, and from this it follows that the 'non-abelian exterior square' of \cite{BL87} is is naturally isomorphic to $[G,\ic{G}]_{\tau(G)}$.
\end{proof}

\begin{remark}\label{remId}
 Using Lemma \ref{lem1}, we  can identify \[G\otimes G=[G,\ic{G}]_{\nu(G)}\quad\text{and}\quad [G,\ic{G}]_{\tau(G)}=G\wedge G\]
  via  $g\otimes h\to [g,\ic{h}]$ and $g\wedge h\to [g,\ic{h}]$, respectively.
\end{remark}

\begin{proposition}\label{prop:stau}
The group $\tau(G)$ is isomorphic to $G^2.(G\wedge G)$ with multiplication
\[(a,b;c)(g,h;d)=(ag,bh;(b^h\wedge g^h)c^{gh}d),\]
and derived subgroup $\tau(G)'\cong (G'\times G').(G\wedge G)$.
\end{proposition}
\begin{proof}
By Lemma \ref{lem1}, the element $g\ic{h}w\in \tau(G)$ corresponds to $(g,h;w)\in G^2.(G \wedge G)$, and this correspondence defines the multiplication in $G^2.(G\wedge G)$. Note that $c\in G\wedge G$ corresponds to an element of the form $\prod_i[x_i,\ic{y_i}]$, and so $c^g$ and $c^{(\ic{g})}$ both correspond to  $\prod_i[x_i^g,\icb{y_i^g}]$. The last claim is \cite[Theorem~3.1]{Roc91}.
\end{proof}

\begin{remark}\label{remCoc1}
  If $G\wedge G$ is abelian, then Proposition \ref{prop:stau} shows that $\tau(G)$ is an extension of $G\wedge G$ by $G^2$ defined by a 2-cocycle $\gamma\in Z^2(G^2,G\wedge G)$ with $\gamma((a,b),(g,h))=b^h\wedge g^h$; the $G^2$-module structure on $G\wedge G$ is defined by  $(u\wedge v)^{(g,h)}=(u^{gh}\wedge v^{gh})$, cf.\ \cite[\S 11.4]{Rob82}.
\end{remark}
\begin{remark}\label{remSplit1}
  The extension in Remark \ref{remCoc1} is split if and only if there is a function $f\colon G^2\to G\wedge G$ such that the subset $\{(a,b;f(a,b)):a,b\in G\}$ is a subgroup of $G^2.(G\wedge G)$ isomorphic to $G^2$ via $(a,b)\mapsto (a,b;f(a,b))$.  In this case $A=\{(a,1;f(a,1)):a\in G\}$ and $B=\{(1,b;f(1,b)):b\in G\}$ are commuting and disjoint subgroups of $G^2.(G\wedge G)$ isomorphic to $G$. In particular, the maps $a\to f(a,1)$ and $b\to f(1,b)$ are $1$-cocycles $G\to G\wedge G$; recall that a 1-cocycle $r\colon G\to G\wedge G$ is a map satisfying $r(gh)=r(g)^hr(h)$ for all $g,h\in G$. Conversely, for every pair of  1-cocycles $l,r\colon G\to G\wedge G$ the sets $L=\{(a,1;l(a)):a\in G\}$ and $R=\{(1,b;r(b)): b\in G\}$ are disjoint subgroups of $G^2.(G\wedge G)$ isomorphic $G$. Together they form a complement to $G\wedge G$ if and only if they commute, that is, if and only if $l(a)^br(b)=(b\wedge a)r(b)^al(a)$ for all $a,b\in G$. The existence of such 1-cocycles is a necessary and sufficient condition for the extension  to be split.
\end{remark} 

\begin{remark}\label{remKappa}
  It follows from  \cite[Proposition 2.5]{BL87} that $G$ acts trivially on the kernels of the maps $\kappa\colon G\wedge G\to G'$ and $\kappa'\colon G\otimes G\to G'$, both induced by the commutator map. This proves that $\ker\kappa\unlhd \tau(G)$ and $\ker\kappa'\unlhd \nu(G)$ are central. Since $(G\wedge G)/\ker\kappa\cong G'$, we have $\tau(G)/\ker\kappa \cong G^2.G'$ with multiplication $(a,b;c)(g,h;d)=(ag,bh;[b,g]^hc^{gh}d)$.   An analysis similar to that in Remark \ref{remSplit1} can be used to determine necessary and sufficient conditions for this extension to be split. 
\end{remark}

\subsection{An explicit description of $\tilde K$}

The following result is based on \cite[Theorem~3.2]{Lie08}.
We denote the components of a tuple $g$ by $g_1,g_2,\ldots$, that is, $g\in G^{n-1}$ is  $g=(g_1,\ldots ,g_{n-1})$. 
\begin{proposition}\label{prop:Kt} Let $G$ be a group with Schur cover $H$ and $H/M=G$. The following hold for $n\geq 3$.
\begin{ithm}
\item We have $K(G,n)\cong G^{n-1}.G'$
where the product of $u=(g;c)$ and $v=(h;d)$ in $G^{n-1}.G'$ is defined as
\[uv=(gh; \mu(g,h)c^hd)\]
where $c^h=c^{(h_1\cdots h_{n-1})^{-1}}$ and 
$\mu(g,h)=(g_1h_1)\cdots (g_{n-1}h_{n-1})(g_1\cdots g_{n-1})^{-1}(h_1\cdots h_{n-1})^{-1}$; we have $\mu(g,h)(c^g)^h=c^{(gh)}\mu(g,h)$ for all $g,h\in G^{n-1}$ and $c\in G'$.
\item Let $\mu$ be the map defining $K(H,n)\cong H^{n-1}.H'$ as in a).  Identifying $H'=G\wedge G$ via the isomorphism in Section \ref{secSM}, we have  $\tilde K(G,n)\cong G^{n-1}.(G\wedge G)$ where the product of  $u=(g;c)$ and $v=(h;d)$ in $G^{n-1}.(G\wedge G)$ is defined as
  \[uv=(gh;\mu(g',h')c^hd);\] here $g',h'\in H^{n-1}$ are elements that map onto $g,h\in G^{n-1}$, and $c^h$ is defined as in a).
\item There is a central extension\[\begin{tikzcd}
    1\arrow{r} & H_2(G,\Z) \arrow{r} & \tilde{K}(G,n) \arrow{r} & K(G,n) \arrow{r} & 1.
\end{tikzcd}\]
\end{ithm}
\end{proposition}
\begin{proof}
\begin{iprf}
\item By definition, $K(G,n)=\{(g_1,\ldots,g_{n-1},g_{n-1}^{-1}\cdots g_1^{-1} c): g_1,\ldots,g_{n-1}\in G, c\in G'\}$. The isomorphism from $G^{n-1}.G'$ to  $K(G,n)$ maps $(g;c)\in G^{n-1}.G'$ to $(g,g_{n-1}^{-1}\cdots g_1^{-1} c)\in K(G,n)$; the definition of $\mu$ and $c^h$ guarantee that this is an isomorphism.
\item   It is shown in \cite[Theorem 3.2]{Lie08} that $\tilde K(G,n)\cong K(H,n)/K(M,n)$, independent of the chosen cover. The proof of a) shows that there is an isomorphism $\varphi\colon H^{n-1}.(G\wedge G)\to K(H,n)$. Recall that $M\leq Z(H)$ is  central, hence it follows from the definition of $\mu$ that $M^{n-1}.1$ is a central subgroup of $H^{n-1}.(G\wedge G)$. This subgroup is mapped under $\varphi$ onto $K(M,n)$, which proves that $\tilde K(G,n)\cong K(H,n)/K(M,n)\cong (H^{n-1}.(G\wedge G))/(M^{n-1}.1)\cong G^{n-1}.(G\wedge G)$. Note that the multiplication is well-defined since $M\leq Z(H)$.
\item This is   \cite[Theorem 2.2]{Lie08}: note that  the epimorphism from $\tilde K(G,n)\cong G^{n-1}.(G\wedge G)$ to $K(G,n)\cong G^{n-1}.G'\cong K(G,n)$ can be induced by  $\kappa\colon G\wedge G\to G'$; recall from Section \ref{secSM} and Remark \ref{remKappa} that $\ker\kappa\cong H^2(G,\Z)$ is central. 
\end{iprf} 
\end{proof}

\begin{remark}
   If $G'$ is abelian, then Proposition \ref{prop:Kt}a) shows that $K(G,n)$ is an extension of $G'$ by $G^{n-1}$ defined by the 2-cocycle $\mu\in Z^2(G^{n-1},G')$ as in the proposition and $G^{n-1}$-module structure on $G'$ defined by  $c^h=c^{(h_1\cdots h_{n-1})^{-1}}$; since $G'$ is abelian, this is indeed a group action. A similar consideration as in Remark \ref{remSplit1} can be used to obtain a (quite technical) criterion for splitness.
\end{remark}

\begin{corollary}\label{cor:nc2}
If $H$ has nilpotency class at most 2, then $K(H,n)\cong H^{n-1}.H'$ with multiplication
\[(g;c)(h,d)=(gh;cd\prod\nolimits_{i=1}^{n-1}\prod\nolimits_{j=i}^{n-1} [g_i,h_j]).\]
\end{corollary}
\begin{proof}
This follows from the formula given in Proposition \ref{prop:Kt}a), together with  $c,d\in H'\leq Z(H)$ and  $[h,g^{-1}]=[h,g^{-1}]^{(g^{h})}=[g,h]$ for all $g,h\in G$.
\end{proof}
 
\subsection{Abelian groups} 
For a group $G$ let $Z^\wedge(G)=\{ g\in G: g\wedge x=1 \hbox{ for all } x\in G\}$ be the \emph{epicentre} of $G$. Note that $Z^\wedge (G)$ is equal to the projection of the center of a Schur cover of $G$ on $G$, see \cite[p.\ 254]{Ell95}, therefore the next result agrees with \cite[Proposition~4.7]{Lie08}. It is shown in \cite[Proposition 16(vii)]{Ell95} that  there exists $H$ with  $H/Z(H)\cong G$ if and only if $Z^\wedge(G)=1$.
\begin{proposition}\label{prop:ab3}
If $G$ is an abelian group, then $\tilde K(G,n)$ is isomorphic to the group $G^{n-1}.(G\wedge G)$ with multiplication
\[(g;c)(h;d)=(gh;cd\prod\nolimits_{i=1}^{n-1}g_i\wedge (h_i\cdots h_{n-1})). \]
Under this identification, 
\begin{eqnarray*} 
Z(\tilde{K}(G,n))&=&\{(u,uy_2,\ldots,uy_{n-1};c)\in G^{n-1}.(G\wedge G): y_2,\ldots,y_{n-1},u^n\in Z^\wedge(G)\}\\[2ex]
&\cong & Z^\wedge(G)^{n-2}\times (G\wedge G)\times \{u\in G: u^n\in Z^\wedge(G)\}.
\end{eqnarray*}
\end{proposition}
\begin{proof}
Let $H$ be a Schur cover of $G$ with $H/M=G$. It follows from Corollary \ref{cor:nc2} and Proposition~\ref{prop:Kt}b) that $\tilde K(G,n)\cong G^{n-1}.H'$ with multiplication
\[(g;c)(h,d)=(gh;cd\prod\nolimits_{i=1}^{n-1}\prod\nolimits_{j=i}^{n-1} [g_i',h_j']),\]
where each $g_i'$ and $h_j'$ is a lift of $g_i,h_j\in G$ to $H$; note that  $H'=M\leq Z(H)$ and $H'=M\cong G\wedge G$ since $G$ is abelian. Recall that  $G\wedge G=\ker c_\tau$, that is, $G\wedge G=\langle g\wedge h : g,h\in G\rangle$ with the convention $g\wedge h=[g,\ic{h}]_{\tau(G)}$. In particular, if $[g',h']_H\in H$ where $g',h'\in H$ are lifts of $g,h\in G$, then  $H'\cong G\wedge G$ via $[g',h']\mapsto g\wedge h$. The first claim follows. 

If  $(a;c)\in Z(\tilde K(G,n))$, then the following holds  for all $(g;d)\in \tilde K(G,n)$:
\begin{eqnarray*}
\prod\nolimits_{i=1}^{n-1} a_i\wedge g_i\cdots g_{n-1} = \prod\nolimits_{i=1}^{n-1} g_i\wedge a_i\cdots a_{n-1}.
\end{eqnarray*}
Considering $g$ with  only one nontrivial entry $g_i=h$, this forces
\[a_1\ldots a_{i-1}a_i^2 a_{i+1}\ldots a_{n-1}\wedge h = 1 \quad\text{for all $h\in G$ and $i\in\{1,\ldots,n-1\}$}.\]
Write $z_i=a_1\ldots a_{i-1}a_i^2 a_{i+1}\ldots a_{n-1}$ and note that each $z_i \in Z^\wedge(G)$; now for $i=2,\ldots,n-1$ we have $z_{i-1}^{-1}z_i=a_{i-1}^{-1}a_i\in Z^\wedge(G)$, so $a_i=a_1y_i$ for some $y_i\in Z^\wedge(G)$. Now $z_1\in Z^\wedge(G)$ yields $a_1^n\in Z^\wedge(G)$. Conversely, it is easy to check that every such element yields a central $(a;c)$.
\end{proof}

\begin{proposition}
\label{prop:tauab}
If $G$ is an abelian group, then $\tau (G)$ is isomorphic to the group $G^2.(G\wedge G)$, where the multiplication is given by $(g_1,g_2;c)(h_1,h_2;d)=(g_1h_1,g_2h_2;cd(g_2\wedge h_1))$. Under this identification,
$Z(\tau (G))=\{ (a,b;c): a,b\in Z^\wedge(G), c\in G\wedge G\}\cong Z^\wedge(G)^2\times (G\wedge G).$
\end{proposition}

\begin{proof}
  The first claim follows from Proposition \ref{prop:stau}. As above, $(a,b;c)\in Z(\tau(G))$ if and only if $b\wedge g=h\wedge a$ for all $g,h\in G$. If $g=1$, then $a\in Z^\wedge(G)$; if $h=1$, then $b\in Z^\wedge(G)$. Conversely, every such $(a,b;c)$ lies in the centre; the claim follows.
\end{proof}

\section{Relating $\tau(G)$ with $\tilde K(G,3)$ and $K(G,3)$}
\label{sec:tau}

The aim of this section is to relate $\tau(G)$ with $\tilde K(G,3)$. As a first step, we consider a construction of an epimorphism $\tau(G)\to K(G,3)$. Our construction requires an automorphism of $G$ which acts as inversion on the abelianisation of $G$.

\subsection{AI-automorphisms}
An automorphism $\alpha\in \Aut(G)$ of a group $G$ is an {\em AI-automorphism} if it induces the inversion automorphism on the abelianisation $G/G'$; this is not to be confused with an {\em IA-automorphism} introduced by Bachmuth (1966), which is an automorphism that induces the identity on the abelianisation. Clearly, the composition of two AI-automorphisms is an IA-automorphism; for abelian groups the only AI-automorphism is inversion.

\begin{example}
  \label{ex:ai}
  Let $F$ be a free group on $X$. The map $X\to X$ given by $x\mapsto x^{-1}$ for all $x\in X$ induces an AI-automorphism $\iota_F$ of $F$. If a group $G$ is given by a free presentation $G=F/N$ such that $\iota_F(N)=N$, then $\iota_F$ induces an AI-automorphism of $G$. Note that if $F/N$ is abelian, then $F'\leq N$, hence $\iota_F(N)=N$ and $\iota_F$ induces inversion on $G$.  If $\iota_F(N)\ne N$, then define $M=\iota_F(N)N\unlhd F$. By definition, $\iota_F(M)=M$, and $F/M$ is the largest quotient of $G$ on which $\iota_F$ induces an AI-automorphism. In particular, every group $G$ has such a quotient since $\iota_F$ induces inversion on $F/F'N\cong G/G'$. We give two examples. First, the dihedral group of order $2n$ can be defined as $D_{2n}=F/N$ where $F$ is free on $\{a,b\}$ and $N$ is the normal closure of $\{a^n,b^2,a^ba\}$. Clearly, $\iota_F(a^n)=(a^{-1})^n$ and $\iota_F(b^2)=b^{-2}$ lie in $N$; moreover, $(\iota_F(a^ba)^{-1})^b=(aa^{b^{-1}})^b=a^ba\in N$, hence $\iota_F$ induces an AI-automorphism on $F/N$. Second, consider $G=F/N$ where $F$ is free on $\{g,h\}$ and $N$ is the normal closure of $\{g^4,h^5,h^gh^2\}$, that is, $G$ is a semidirect product $C_4\ltimes C_5$. A direct computation (by hand or with GAP \cite{gap}) shows that $G$ does not admit an AI-automorphism, which implies that  $\iota_F(N)\ne N$. If $M$ is the normal closure of $\{g^4,h^5,(h^{-1})^{(g^{-1})}h^{-2},h^gh^2\}$, then $\iota_F(M)=M$, and  $G/M\cong C_4$ is the largest quotient of $G$ on which $\iota_F$ induces an AI-automorphism.
  \end{example}

\begin{example}
Let $\alpha\in\Aut(G)$ be an automorphism which inverts every element of a generating  set $X$ of $G$. Such an automorphism is called  {\em GI-automorphism} in \cite{Bos}, where GI can be interpreted as ``generator inversion''. (Originally, GI stands for ``generator-involutions'' because  $\langle\alpha\rangle\ltimes G$ is generated by involutions $\{(\alpha,x): x\in X\}$.) Clearly, every GI-automorphism is an AI-automorphism. The map $\iota_F$ in Example \ref{ex:ai} is an example. To give another example, consider the alternating group $\text{Alt}_n$ of rank $n\geq 3$: Conjugation by the transposition $(1\, 2)$ defines an automorphism $\alpha$ of $\text{Alt}_n$ that inverts every element of the generating set $\{(1\, 2\, 3),(1\, 2\, 4),\ldots,(1\, 2\, n)\}$; thus $\alpha$ is a GI- and AI-automorphism.
\end{example}

\subsection{An epimorphism}\label{secEpi} Suppose $G$ has an AI-automorphism $\alpha$; we use $\alpha$  to construct $K(G,3)$ as a quotient of $\tau(G)$. Note that the homomorphism
\[ G\star \ic{G}\to G^3,\quad g_1\ic{h_1}\ldots g_k\ic{h_k}\mapsto (g_1\ldots g_k,h_1\ldots h_k,\alpha(g_1h_1\ldots g_kh_k))\]
maps commutators $[x,\ic{x}]$ to 1; since the above map forgets ``$\ast$'', it also maps the relations of $\tau(G)$ to~1. Thus there is an induced homomorphism
\[ {\Phi_{\alpha}}\colon \tau(G)\to G^3.\]

Recall that the commutator map \[\kappa\colon G\wedge G=[G,\ic{G}]_{\tau(G)}\to G'\] has central kernel $H_2(G,\mathbb{Z})\cong M(G)$, see Section \ref{secSM}. We now show the following:

\begin{theorem}\label{thm:kernel}
  If $\alpha\in\Aut(G)$ is an AI-automorphism, then \[\im {\Phi_\alpha}=K(G,3)\quad\text{and}\quad \ker {\Phi_\alpha}=\ker\kappa\leq Z(\tau(G)).\]
\end{theorem}
\begin{proof}The inclusion $\im {\Phi_\alpha}\leq K(G,3)$  follows immediately from the definition and the fact that $\alpha$ is an AI-automorphism. If $(g,h,k)\in K(G,3)$, then $k=h^{-1}g^{-1}c$ for some $c\in G'$. Note that $ {\Phi_\alpha}$ maps $g\ic{h}$ to $(g,h,\alpha(gh))\in K(G,3)$, and  $\alpha(gh)=h^{-1}g^{-1}d$ for some $d\in G'$, thus
\[ {\Phi_\alpha}(g\ic{h})^{-1}\cdot (g,h,k)=(1,1,d^{-1}c);\]
now $d^{-1}c=\prod_i[x_i,y_i]\in G'$, and so $(1,1,d^{-1}c)= {\Phi_\alpha}(\prod_i[\alpha^{-1}(x_i),\icb{\alpha^{-1}(y_i)}])$. This shows that $(g,h,k)\in \im {\Phi_\alpha}$, thus $K(G,3)\leq \im {\Phi_\alpha}$. Now we consider the kernel. Note that \[\ker {\Phi_\alpha}=\{g_1\ic{h_1}\ldots g_k\ic{h_k}: g_1\cdots g_k=h_1\cdots h_k=(g_1h_1)\cdots(g_kh_k)=1\}.\]If $w=g_1\ic{h_1}\ldots g_k\ic{h_k}\in\ker {\Phi_\alpha}$, then use  Lemma \ref{lem1} to  rewrite $w=g_1\cdots g_k\icb{h_1\cdots h_k} w'=w'$ for some $w'=\prod_i[x_i,\ic{y_i}]\in [G,\ic{G}]_{\tau(G)}$; thus, $w=w'\in G\wedge G$ and  applying $\kappa$ yields $\kappa(w)=\kappa(w')=\prod_i [x_i,y_i]\in G$. Note that rewriting $w$ to $w'$ involves a sequence of commutator rules as in \eqref{eqIds}, replacing elements such as $\ic{a}b$ by $b\ic{a}[\ic{a},b]$, etc.  Obviously, this rewriting process can be reversed which yields a sequence of commutator rules that bring $w'$ back into the form  $w$. Applying this reversed process not to $w'$, but to $\prod_i [x_i,y_i]_G$, we obtain $g_1h_1\ldots g_kh_k\in G$, which is the element $w$ without the ``$\ast$''. Recall that $g_1h_1\ldots g_kh_k=1$ by assumption, which shows that  \[\kappa(w)=\kappa(w')=\prod\nolimits_i [x_i,y_i]=g_1h_1\ldots g_kh_k=1,\]so $w\in\ker\kappa$. Conversely, let $w\in\ker\kappa$, that is, $w=\prod_i[g_i,\ic{h_i}]\in [G,\ic{G}]_{\tau(G)}$ with $\prod_i [g_i,h_i]=1$. Writing $w$ as $w=\prod_i g_i^{-1}\icb{h_i^{-1}}g_i\ic{h_i}$ and applying $ {\Phi_\alpha}$ shows that \[{\Phi_\alpha}(w)=(1,1,\alpha([g_1,h_1]\ldots [g_k,h_k]))=(1,1,1),\] hence $\ker\kappa\leq \ker {\Phi_\alpha}$. In conclusion, $\ker {\Phi_\alpha}=\ker\kappa$. By  Remark~\ref{remKappa}, the group $\ker\Phi_\alpha$ is central.
\end{proof}

\begin{corollary}\label{cor:seq}
The existence of an AI-automorphism of $G$ yields  a central extension
\[\begin{tikzcd}
    1\arrow{r} & H_2(G,\Z) \arrow{r} & \tau(G) \arrow{r} & K(G,3) \arrow{r} & 1.
\end{tikzcd}\]
\end{corollary}
Together with Proposition \ref{prop:Kt}c), it seems natural to ask when $\tau(G)\cong \tilde K(G,3)$. We will see in Proposition \ref{propAllSylowCyclic} that the lack of AI-automorphisms may prevent this.

\subsection{A subgroup $T(G)$}\label{secTG}
Recall that $K(G,3)$ is the kernel of $G^3\to G/G',\; (g,h,k)\mapsto ghkG'$. We now consider the kernel
\[T(G)=\{(g,h,cgh): g,h\in G, c\in G'\}\]
of the homomorphism $\pi\colon G^3\to G/G'$, $(g,h,k)\mapsto ghk^{-1}G'$. We now provide a short alternative proof that $K(G,3)\cong \tau(G)/\ker\kappa$ if $G$ has an AI-automorphism, cf.\ Theorem \ref{thm:kernel}.

 \smallskip

\begin{lemma}\label{lemTG} 
  We have $T(G)\cong \tau(G)/\ker\kappa$. If $G$ has  AI-automorphisms, then $K(G,3)\cong T(G)$.
\end{lemma}
\begin{proof}
Recall that we  can identify  $\tau(G)/\ker\kappa=G^2.G'$  and  $K(G,3)=G^2.G'$ via Remark~\ref{remKappa} and  $(g,h,h^{-1}g^{-1}c) = (g,h;c)$, respectively, with the following multiplications
 \begin{align*} K(G,3): &\quad (a,b;c)(g,h;d)=(ag,bh;[a^{-1},g^{-1}][b^{-1}a^{-1},h^{-1}]^{g^{-1}}c^{(gh)^{-1}}d)\text{ and}\\
   \tau(G)/\ker\kappa: &\quad(a,b;c)(g,h;d)=(ag,bh;[b,g]^hc^{gh}d).
 \end{align*}
 Every element in $T(G)$ can be written as $(a,b,abc)$ for unique $a,b\in G$ and $c\in G'$. This allows us to identify $T(G)=G^2.G'$ via $(a,b,abc)=(a,b;c)$, and a short calculation confirms that the induced multiplication in $T(G)=G^2.G'$ is the same as for $\tau(G)/\ker\kappa=G^2.G'$, so $T(G)\cong \tau(G)/\ker\kappa$. 
If $\alpha$ is an AI-automorphism of $G$, then $\beta=1\times 1\times \alpha$ is an automorphism of $G^3$ that interchanges $K(G,3)$ and $T(G)$, hence $T(G)\cong K(G,3)$.
\end{proof}
Identifying $T(G)=\tau(G)/\ker\kappa$, the isomorphism $\beta\colon T(G)\to K(G,3)$ in the proof of Lemma \ref{lemTG} coincides with the isomorphism $\tau(G)/\ker G\cong K(G,3)$ induced by $\Phi_\alpha$ in Theorem \ref{thm:kernel}.

\section{Some isomorphisms}
\label{sec:isom}

\noindent
Our computations in Section \ref{sec:comput} suggest that $\tau(G)\cong \tilde K(G,3)$ only if $G$ admits an AI-automorphism, cf.\ Corollary  \ref{cor:seq}. As mentioned above, the lack of AI-automorphisms may prevent isomorphisms, but one may ask whether an AI-automorphism implies $\tau(G)\cong \tilde K(G,3)$. In general, the answer is no, as illustrated by Proposition \ref{prop:rank2}b) and Examples \ref{exNonI} and \ref{ex:cf}. However, there is strong evidence that $\tau(G)$ is closely related to $\tilde K(G, 3)$ when AI-automorphisms exists; the next theorem is a useful tool for establishing various isomorphisms.

\begin{theorem}\label{thm:epi2}
  Suppose $G$ has an AI-automorphism that lifts to an AI-automorphism of a Schur cover inverting the Schur multiplier.
  \begin{ithm}
   \item We have $\tilde{K}(G,3)\cong \tau(G)$.
   \item We have $\tilde K(G,3)\cong T(H)/T(M)$.
  \end{ithm}
\end{theorem}
\begin{proof}
\begin{iprf}  
\item Let $H$ be a Schur cover with $H/M=G$ and let $\alpha\in \Aut(H)$ be the induced AI-automorphism with $\alpha(m)=m^{-1}$ for all $m\in M$. Corollary \ref{cor:seq} shows that  $ {\Phi_\alpha}\colon \tau(H)\to K(H,3)$ is an epimorphism with kernel $H_2(H,\Z)$. It is shown in \cite[Theorem 3.2]{Lie08} that $\tilde K(G,n)$ is isomorphic to $K(H,n)/K(M,n)$, so we obtain an epimorphism $\tau(H)\to\tilde K(G,3)$. By Lemma \ref{lem:kerRoc},  the projection $H\to G$ yields a surjection $\gamma\colon \tau(H)\to\tau(G)$ with kernel  $(\langle M, \ic{M}\rangle[M,\ic{H}][H,\ic{M}])_{\tau(H)}$. We can construct an induced epimorphism $\tau(G)\to\tilde{K}(G,3)$ if $ {\Phi_\alpha}(\ker\gamma) \leq K(M,3)$. If $m\in M$, then $ {\Phi_\alpha}(m)=(m,1,\alpha(m))$, which lies in $K(M,3)$ since $\alpha(m)=m^{-1}$; similarly for $\ic{m}\in \ic{M}$. If $[m,\ic{h}]$ is a generator of $[M,\ic{H}]$, then this is mapped under $ {\Phi_\alpha}$ to $(1,1,\alpha([m,h]))=(1,1,1)$ since $M\leq Z(H)$; similarly for elements in $[H,\ic{M}]$. Thus, $ {\Phi_\alpha}(\ker\gamma) \leq K(M,3)$ and the claim follows.
\item We have $T(M)=\{(a,b,ab):a,b\in M\}$ and $K(M,3)=\{(a,b,a^{-1}b^{-1}):a,b\in M\}$. The isomorphism $T(H)\cong K(H,3)$ of Lemma \ref{lemTG} maps $T(M)$ onto $K(M,3)$; recall that $\alpha$ inverts $M$. This implies that  $\tilde K(G,3)\cong K(H,3)/K(M,3)\cong T(H)/T(M)$.
\end{iprf}
\end{proof}

\begin{remark}
  \begin{iprf} 
    \item If $G$ has an abelian Schur cover, say $H/M=G$, then $M\leq H'$ implies that $M=1$, so $G=H$ is cyclic, the assumptions of Theorem \ref{thm:epi2} are satisfied, and $\tau(G)\cong \tilde{K}(G,3)\cong K(G,3)$. 
   \item If a Schur cover $H$ of $G$ admits an AI-automorphism $\alpha$ that leaves $M$ invariant, then $\alpha$ induces an AI-automorphism of $G\cong H/M$ since $H/H'\cong G/G'$.   
   \item Based on Theorem \ref{thm:epi2} and example computations,  we conjecture that $\tau(G)\cong \tilde K(G,3)$ only if $G$ admits an AI-automorphism. The results that follow and Example \ref{ex:cf} below support this conjecture. A stronger conjecture would be that  $\tau(G)\cong \tilde K(G,3)$ if and only if  $G$ admits an AI-automorphism that lifts to an AI-automorphism of a Schur cover inverting the Schur multiplier. However, this is not true as can be shown by a direct computation with GAP: the group $G=C_4\times C_4$ has Schur multiplier $M\cong C_4$, has an AI-automorphism, and satisfies $\tau(G)\cong\tilde K(G,3)$; up to isomorphism $G$ has three Schur covers $H_1$, $H_2$, and $H_3$, with GAP SmallGroup id [64,18], [64,19], and [64,28], respectively. Each $H_i$ has a unique $M_i\leq H_i'\cap Z(H_i)$ with $M_i\cong M$ and $H_i/M_i\cong G$. Only $H_1$ and $H_2$ have AI-automorphisms, but all of those act trivially on $M$. A similar statement holds for the non-abelian $C_4\times (C_4\ltimes C_3)$ with GAP id [48,11]. This illustrates several things: First, whether or whether not an AI-automorphism of $G$ lifts to an AI-automorphism of a Schur cover depends on the isomorphism type of the Schur cover. Second, we can have $\tau(G)\cong \tilde K(G,3)$ even though there is no lift of an AI-automorphism of $G$ that inverts the Schur multiplier. 
  \end{iprf}
\end{remark}
 
\begin{corollary}\label{corExp2}
If $G$ is a group with $\exp(G/G'),\exp(M(G))\in\{1,2\}$, then  $\tau(G)\cong \tilde K(G,3)$.
\end{corollary}
\begin{proof}
  Let $H$ be a Schur cover of $G$ with $H/M\cong G$. Since $H/H'\cong G/G'$ and $\exp(M)\mid 2$,  the identity automorphism is an AI-automorphism inverting $M$.  Now  Theorem \ref{thm:epi2} proves the claim.
\end{proof}

The next result considers  the finite groups all whose Sylow subgroups are cyclic, see \cite[10.1.10]{Rob82}. Note that every group of square-free order has this property.
 
\begin{proposition}\label{propAllSylowCyclic}
 Let $G$ be a group all whose Sylow subgroups are cyclic, that is,  \[G=\langle a,b\mid b^n, a^m, a^b=a^r\rangle\cong C_n\ltimes C_m\] where $|G|=mn$ with $m$ odd, and $0\leq r<m$ with $r^n\equiv 1\bmod m$  and $\gcd(m,n(r-1))=1$. Then $G$ has trivial Schur multiplier, hence $\tilde K(G,3)=K(G,3)$, and the following hold.
  \begin{ithm}
  \item The group $G$ has AI-automorphisms if and only if $r^2\equiv 1\bmod m$.
   \item  If $G$ is square-free, then $G$ has AI-automorphisms if and only if $G$ has a cyclic $2'$-Hall subgroup.
  \item The group $G$ satisfies $\tilde K(G,3)\cong \tau(G)$ if and only if $G$ has AI-automorphisms.
  \end{ithm} 
\end{proposition}
\begin{proof}
 It follows from H\"older's classification \cite[10.1.10]{Rob82} that the finite groups all whose Sylow subgroups are cyclic are exactly the groups having a presentation as in the proposition.  It follows from \cite[Corollary 2.1.3]{Kap} that $G$ has trivial Schur multiplier, hence $\tilde K(G,3)=K(G,3)$ by definition. If $G$ is abelian, then $G$ is cyclic and Theorem \ref{thm:epi2} proves the claim where the AI-automorphism is inversion. The condition $\gcd(m,n(r-1))=1$ guarantees that $r\ne 1$, hence $G$ is abelian if and only if $r=0$ and $m=1$; note that in this case $r^2\equiv 1\bmod m$ holds trivially. Thus, in the following we  assume that $G$ is non-abelian, that is, $r>1$.
   \begin{iprf}
   \item  Note that $[a,b]=a^{r-1}$, so  $G'=\langle a\rangle$. If $G$ has an AI-automorphism, then there exist $u,v$ with $\gcd(u,m)=1$ such that $b^{-1}a^v$ and $a^u$ satisfy the relations of $b$ and $a$ in $G$. The conjugacy relation forces $a^b=a^{b^{-1}}$, that is, $r^2\equiv 1\bmod m$. Conversely, if $r^2\equiv 1\bmod m$, then  $(b,a)\mapsto (b^{-1},a)$ describes an AI-automorphism of $G$. 

 \item If $G$ is square-free with cyclic $2'$-Hall subgroup $V\cong C_{nm/2}$, then there is a subgroup $U\cong C_2$ with  $G=U\ltimes  V\cong C_2\ltimes C_{mn/2}$, see \cite[Ex.~1.3(13) and (9.1.2)]{Rob82}. In particular, $V$ is the unique Hall $2'$-subgroup, which shows that $V=\langle b^2,a\rangle$ and we can choose $U=\langle b^{n/2}\rangle$. Thus, by renaming the generators, we can assume that  $G=\langle a,b\mid a^m, b^2, a^b=a^r\rangle$ where $r^2\equiv 1\bmod m$, $m$ is odd, $0\leq r<m$, and $\gcd(m,r-1)=1$. Now by part a), the identity defines an AI-automorphism of $G$.  Conversely, if $G$ is square-free with AI-automorphisms, then a) implies that $G\cong \langle b^{n/2}\rangle \ltimes \langle b^2,a\rangle\cong C_2\ltimes C_{mn/2}$, so $G$ has a cyclic Hall $2'$-subgroup.

\item If $G$ has an AI-automorphism, then Theorem \ref{thm:epi2} proves that $\tau(G)\cong\tilde K(G,3)\cong K(G,3)$; recall that $M(G)=1$. Conversely, suppose that  $\tau(G)\cong \tilde K(G,3)=K(G,3)$; abbreviate  $T=\tau(G)$ and $K=K(G,3)$.
  If we interpret $T$ via Proposition \ref{prop:stau}, we get generators
  $y_1 = (b,1,1)$,
  $x_1 = (a,1,1)$,
  $y_2 = (1,b,1)$,
  $x_2 = (1,a,1)$,
  and $x_3 = (1,1,a)$, and it follows that $T'=\langle x_1,x_2,x_3\rangle\cong C_m^3$ and $T/T'=\langle y_1T',y_2T'\rangle$. The elements $y_iT'$  act on $T'$ from the right via  matrices
  $$m_1=\left ( \begin{smallmatrix}
  r & 0 & 0\\ 0 & 1 & r-1 \\ 0 & 0 & r\end{smallmatrix} \right )\quad\text{and}\quad m_2=\left ( \begin{smallmatrix}
    1 & 0 & r-1\\ 0 & r & 0 \\ 0 & 0 & r\end{smallmatrix} \right ),$$
    both given with respect to $x_1,x_2,x_3$. Similarly, $K$ is generated by
$\tilde{y}_1 = (b,1,b^{-1})$,
$\tilde{x}_1 = (a,1,1)$,
$\tilde{y}_2 = (1,b,b^{-1})$,
$\tilde{x}_2 = (1,a,1)$, and
$\tilde{x}_3 = (1,1,a)$, and it follows that  
$K'=\langle \tilde{x}_1,\tilde{x}_2,\tilde{x}_3\rangle\cong C_m^3$, and $K/K'=\langle \tilde{y}_1K',\tilde{y}_2K'\rangle$. Here the elements $\tilde{y}_i K'$ act  on $K'$ from the right via the matrices
  $$n_1=\left ( \begin{smallmatrix}
  r & 0 & 0\\ 0 & 1 & 0 \\ 0 & 0 & s\end{smallmatrix} \right )\quad\text{and}\quad
  n_2=\left ( \begin{smallmatrix}
  1 & 0 & 0\\ 0 & r & 0 \\ 0 & 0 & s\end{smallmatrix} \right ),$$
  where $s$ is the multiplicative inverse of $r$ modulo $m$.
  Now consider the subgroups $A=\langle m_1,m_2\rangle$ and $B=\langle n_1,n_2\rangle$ of $\GL_3(m)$. As $T$ and $K$ are isomorphic, it follows that $A$ and $B$ are conjugate in $\GL_3(m)$. Since $B$ is contained in $\SL_3(m)$, the same holds for $A$. This forces $r^2\equiv 1 \bmod m$, and now  part a) shows that $G$ admits an AI-automorphism.
  \end{iprf}
\end{proof}

\begin{proposition}\label{prop:es}
  Let $G$ be an extra-special $p$-group with $p$ odd.
  \begin{ithm}
  \item Let $\exp(G)=p$. If $|G|=p^3$, then  $\tau(G)\cong \tilde K(G,3)$; if $|G|=p^{2n+1}$ with $n\geq 2$, then there exist Schur covers of $G$ that admit AI-automorphisms, but none of these inverts the Schur multiplier. 
  \item If $\exp(G)=p^2$, then $G$ does not have an AI-automorphism.
  \item If $|G|=p^3$ and $\exp(G)=p^2$, then $\tau(G)\not\cong \tilde K(G,3)$.  
  \end{ithm} 
\end{proposition}
\begin{proof} 
  \begin{iprf}
  \item Let $G$ be an extra-special $p$-group of exponent $p$ and order $p^{2n+1}$. It follows from \cite[Satz~III.13.7]{Hup} that $G$ is a central product of $n$ extra-special groups of size $p^3$ and exponent $p$, that is, we can assume that $G={\rm pc}\langle g_1,\ldots,g_{2n},c\mid \forall i,j: [g_{2i},g_{2i-1}]=c^{-1},\; g_{j}^p=c^p=1\rangle$. First suppose that $n=1$. By  \cite[Theorem 3.3.6]{Kap}, the Schur multiplier is isomorphic to $C_p^2$, and it is straightforward to verify that the group\[H={\rm pc}\langle g_1,g_2,c,h_1,h_2\mid g_1^p,g_2^p,c^p,h_1^p,h_2^p, [g_2,g_1]=c^{-1}, [c,g_1]=h_1,[c,g_2]=h_2\rangle,\]
is a Schur cover of $G$ with $H/M=G$ for  $M=\langle h_1,h_2\rangle\cong C_p^2$. The elements $g_1^{-1}c,$ $g_2^{-1}c^{-1}$, $c$, $h_1^{-1}$, $h_2^{-1}$ satisfy the relations of $H$, so von Dyck's Theorem \cite{Rob82}*{2.2.1} shows that  $(g_1,g_2,c,h_1,h_2)\mapsto (g_1^{-1}c, g_2^{-1}c^{-1},c, h_1^{-1}, h_2^{-1})$ extends to an automorphism $\alpha$ of $H$. This is an AI-automorphism of $H$ that inverts elements of $M$, so  Theorem \ref{thm:epi2} proves the claim for $n=1$.

Now let $n>1$. Beyl and Tappe (1982) \cite[Theorem 3.3.6]{Kap} proved that  $M=M(G)$ is elementary abelian of rank $2n^2-n-1$ and that every Schur cover $H$ of $G$ with $H/M=G$ is unicentral, that is, $Z(H)$ is the full preimage of $Z(G)$ under the projection $H\to G$; in particular, we have $Z(G)=G'=H'/M$ and  $H'=Z(H)$.
Thus, if $g,h\in H$, then $\alpha([g,h])=[\alpha(g),\alpha(h)]=[g^{-1},h^{-1}]=[g,h]$, so $\alpha$ fixes $H'$ (and so $M\leq H'$) elementwise; in particular, $\alpha$ does not invert $M$.

An explicit Schur cover $H$ of $G$ can be defined by abstract generators $g_1,\ldots,g_{2n},c$ and $h_{i,j}$ for $1\leq i<j\leq n$ except $(i,j)=(1,2)$, all of order $p$, with each $h_{i,j}$ and $c$ central, and the following nontrivial commutator relations:  each commutator relation $[g_j,g_i]=w$ in $G$ with $i<j$ (except for $[g_2,g_1]$) becomes a relation $[g_i,g_j]=wh_{i,j}$ in $H$. Let $N$ be the subgroup  generated by all $h_{i,j}$; it follows from the construction that $N\leq Z(H)\cap H'$, that $Z(H)=\langle c,N\rangle$, and that  $H/N\cong G$. Standard {\it consistency checks} (see \cite[Section~8.7.2]{handbook}) can be used to show that this presentation is {\it consistent}: since every element has order $p$, the only tests that have to be carried out are for the equations $(g_ig_j)g_k=g_i(g_jg_k)$ with $k<j<i$, but all those lead to the conditions $h_{j,i}h_{k,i}h_{k,j}=h_{k,i}h_{j,i}h_{k,j}$ which are trivially satisfied. Consistency of the presentation implies $|H|=p^{2n+1+2n^2-n-1}$, so $H/N\cong G$ proves that  $N\cong C_p^{2n^2-n-1}$ is isomorphic to the Schur multiplier. This shows that $H$ is indeed a Schur cover of $G$. In particular, AI-automorphisms exist: take the isomorphism that is defined by mapping each generator $g_i$ to $g_i^{-1}$.
\item Let $G$ be  extra-special of order $p^{1+2n}$ with $Z(G)=\langle c\rangle=G'$. It follows from \cite[Satz~III.13.7]{Hup} that $G$ is a central product of $n$ extra-special groups of size $p^3$, at least one of them of exponent $p^2$. Thus, there are  $g,h\in G$ such that $\langle g,h,c\rangle$ is extraspecial of order $p^3$ and exponent $p^2$;  we can assume that $g^p=c$, $h^p=c^p=1$, and $[h,g]=c^{-1}$. If $\alpha\in\Aut(G)$ is an AI-automorphism, then $\alpha(c^{-1})=[\alpha(h),\alpha(g)]=[h^{-1},g^{-1}]=[h,g]^{(-1)^2}=c^{-1}$, so $\alpha(c)=c$. Now if $\alpha(g)=g^{-1}d$ with $d\in Z(G)$, then $c=\alpha(g)^p=g^{-p}d^p=g^{-p}=c^{-1}$ forces $|c|=2$, a contradiction.
\item We consider $G=\langle g,h,c\mid g^p=c, h^p, c^p , [h,g]=c^{-1}, [c,g], [c,h]\rangle$ with $M(G)=1$. Recall that $\tilde K(G,3)=K(G,3)$ and $\tau(G)\cong T(G)$, see Lemma \ref{lemTG}. We show that $T=T(G)$ and $K=K(G,3)$ are not isomorphic, which proves the claim. For $i\in\{1,2,3\}$ and $x\in G$ let $x_i$ be the element $x$ in the $i$-th copy of $G^3$. Note that $K$ is generated by $\{x_1x_3^{-1},x_2x_3^{-1},c_1,c_2,c_3: x\in G\}$, whereas $T$ is generated by $\{x_1x_3, x_2x_3, c_1,c_2,c_3: x\in G\}$.  One can show that $Z(K)=\langle c_1,c_2,c_3\rangle=Z(T)$ and $\mho(K)=\langle c_1c_3^{-1},c_2c_3^{-1}\rangle$ and $\mho(T)=\langle c_1c_3,c_2c_3\rangle$. Moreover, $\Omega(T)=\langle h_1h_3,h_2h_3,c_1,c_2,c_3\rangle$ and $\Omega(K)=\langle h_1h_3^{-1},h_2h_3^{-1},c_1,c_2,c_3\rangle$. Recall that $\mho$ and $\Omega$ denote the subgroups generated by $p$-th powers and elements of order $p$, respectively. Let $\mathcal{B}=\{c_1c_3^{-1},c_2c_3^{-1},c_3\}$ be a basis of the $\Z_p$-space $Z(K)$. The commutator map $K/\Omega(K)\times \Omega(K)/Z(K)\to Z(K)$ is induced by $[g_1g_3^{-1},h_1h_3^{-1}]=c_1c_3$, $[g_1g_3^{-1},h_2h_3^{-1}]=c_3$, $[g_2g_3^{-1},h_1h_3^{-1}]=c_3$, and $[g_2g_3^{-1},h_2h_3^{-1}]=c_2c_3$. Moreover, the power map $K/\Omega(K)\to \mho(K)\leq Z(K)$ satisfies $(g_1g_3^{-1})^p=c_1c_3^{-1}$ and $(g_2g_3^{-1})^p=c_2c_3^{-1}$. Note that with respect to $\mathcal{B}$, the element $c_1c_3$ is represented as $(1,0,2)$, etc. All together, these commutator and power maps are encoded by the  $\Z_p$-matrix $\mathcal{M}(K)$ below; analogously, we obtain the matrix $\mathcal{M}(T)$ with respect to the basis $\mathcal{B}'=\{c_1c_3,c_2c_3,c_3\}$ of the $\Z_p$-space $Z(T)$:
    \[\mathcal{M}(K)=\left(\begin{smallmatrix}
    1 &0 & 2  \\
    0 &0 & 1  \\
    0 &0 & 1  \\
    0 &1 & 2 
\end{smallmatrix}\middle|
\begin{smallmatrix}
   1 & 0 &0 \\
   1 & 0 &0 \\
    0 & 1 &0 \\
    0 & 1 &0 
  \end{smallmatrix}\right)\quad\quad \mathcal{M}(T)=\left(\begin{smallmatrix}
    1 &0 & 0 &  \\
    0 &0 & 1 & \\
    0 &0 & 1 &  \\
    0 &1 & 0 & 
  \end{smallmatrix}\middle|\begin{smallmatrix}
    1 & 0 &0 \\ 
    1 & 0 &0 \\
   0 & 1 &0 \\
   0 & 1 &0 
  \end{smallmatrix}\right).
  \]If $K\cong T$, then there exist $A,B\in\GL_2(p)$ and $C\in\GL_3(p)$ with $(A\otimes B)\cdot \mathcal{M}(T)=\mathcal{M}(K)\cdot (I_2\otimes C)$; since  $\mho(T)$ and $\mho(K)$ are characteristic subgroups, $C$ has entry $0$ in position $(1,3)$ and $(2,3)$. A straightforward but technical calculation shows that such $A,B,C$ cannot exist, thus, $T\not\cong K$.  
  \end{iprf} 
\end{proof}

\begin{remark}
  In general, deciding (non)-isomorphism for $\tau(G)$ and $\tilde K(G,3)$ seems to be an intricate matter since already for extra-special $G$ of order $3^5$, both $\tau(G)$ and $\tilde K(G,3)$ are extensions of $C_3^8$ by $C_3^8$. As explained in Section \ref{sec:comput}, even advanced computational group theory methods fail for such isomorphism tests. 
\end{remark}

Next, for $n\geq 1$ we consider the generalised quaternion group $Q_{4n}$ and dihedral group $D_{2n}$ of order $4n$ and $2n$, respectively, which are defined as
\[Q_{4n}=\langle a,b\mid a^{2n}, b^2=a^n, a^b=a^{-1}\rangle\quad\text{and}\quad D_{2n}=\langle a,b\mid a^{n}, b^2, a^b=a^{-1}\rangle.\]

\begin{proposition}\label{prop:Qn}
We have $\tau(Q_{4n})\cong\tilde K(Q_{4n},3)$ and $\tau(D_{2n})\cong \tilde K(D_{2n},3)$.
\end{proposition}
\begin{proof}
For $Q_4=C_4$ and $D_2=C_2$ the claim is obvious, so let $n\geq 2$.  It follows from \cite[Example 2.4.8]{Kap} that $M(Q_{4n})=1$. Note that $\{a^{-1},b^{-1}\}$ also satisfies the relations of $Q_{4n}$, so $(a,b)\mapsto (a^{-1},b^{-1})$ extends to a GI-automorphism of $Q_{4n}$ by von Dyck's Theorem. Now $\tau(Q_{4n})\cong \tilde K(Q_{4n},3)$ by Theorem \ref{thm:epi2}. Let $H$ be a Schur cover of $D_{2n}$ with $H/M=D_{2n}$. By \cite[Proposition 2.11.4]{Kap}, we have $M=1$ and $H=D_{2n}$ if $n$ is odd, and $M=C_2$ and $H=Q_{4n}$ otherwise. As seen above and in Example~\ref{ex:ai}, the group $H$ admits an AI-automorphism which fixed $M$ elementwise.  Again, the claim follows with  Theorem \ref{thm:epi2}.
\end{proof}

\begin{proposition}
We have $\tau({\rm Sym}_n)\cong \tilde K({\rm Sym}_n,3)$ and
$\tau({\rm Alt}_n)\cong \tilde K({\rm Alt}_n,3)$.
\end{proposition}
\begin{proof}
  For $n\leq 3$ the claim can be verified directly, so let $n\geq 4$ in the following.   Schur (1911) proved that the Schur multiplier of ${\rm Sym}_n$ is cyclic of order 2 for $n\geq 4$, and trivial otherwise, see  \cite[Theorem 2.12.3]{Kap}. Now Corollary \ref{corExp2} proves the claim for ${\rm Sym}_n$. Similarly, if $n\notin\{4,6,7\}$, then the claim for ${\rm Alt}_n$ follows from \cite[Theorem 2.12.5]{Kap} and Corollary \ref{corExp2} for ${\rm Alt}_n$. The case $n=4$ can be checked directly, and if $n\in\{6,7\}$, then the outer automorphism extending ${\rm Alt}_n$ to ${\rm Sym}_n$ inverts the Schur multiplier: this also follows directly from the presentations given in    \cite[Theorem 2.12.5]{Kap}.
\end{proof}

The next result shows that Theorem \ref{thm:epi2} cannot be applied to abelian groups $G$ in general. Recall that if $M$ is a trivial $G$-module of exponent $2$, then a 2-coboundary $\delta\in B^2(G,M)$ is a function $G\times G\to M$ defined by a map $\kappa\colon G\to M$ with $\kappa(1)=1$ such that $\delta(g,h)=\kappa(gh)\kappa(g)\kappa(h)$ for all $g,h\in G$. In the following, for an abelian group $G$, we write $G=G_2\times G_{2'}$ where $G_2$ is the Sylow $2$-subgroup of~$G$.

\begin{proposition}\label{propabc}
Let $G$ be an abelian group with Schur cover $H$, say $H/M=G$. Then $H$ admits an AI-automorphism whose restriction to $M$ is inversion if and only if $G_{2'}$ is cyclic, $M$ has exponent dividing $2$, and the map $G_2\times G_2\to G_2\wedge G_2$ defined by $(g,h)\mapsto g\wedge h$ is a 2-coboundary; in particular, any such AI-automorphism has order dividing 2.
\end{proposition}
\begin{proof}
  First suppose that $H$ admits an AI-automorphism, say $\alpha$, whose restriction to $M$ is inversion.  Since $G$ is abelian, $H'\leq M$, and  now $M\leq H'\cap Z(H)$ implies $M=H'\leq Z(H)$. We decompose $G=G_2\times G_{2'}$ as above. It follows from \cite[Lemma 2.9.1]{Kap} that the Schur cover $H$ of $G$ is the direct product of Schur covers of $G_2$ and $G_{2'}$, respectively. Thus, we first assume that $G=G_{2'}$ and show that $G$ is cyclic. By assumption, for every $h\in H$ we can write $\alpha(h)=h^{-1}c_h$ for some $c_h\in H'$. Now \[h^{-1}g^{-1}c_{gh}=\alpha(gh)=\alpha(g)\alpha(h)=g^{-1}c_g h^{-1}c_h=h^{-1}g^{-1}[g^{-1},h^{-1}]c_gc_h\]implies that $c_{gh}=[g^{-1},h^{-1}]c_gc_h$ for all $g,h\in H$. Note that $[g,h]=[g^{-1},h^{-1}]^{gh}=[g^{-1},h^{-1}]$ since $H'$ is central, so $c_{gh}=c_gc_h[g,h]$. Moreover, $1=c_1=c_{gg^{-1}}$ yields $c_{g^{-1}}=(c_g)^{-1}$. This can be used to show that $\alpha^{2n+1}(g)=g^{-1}c_g^{2n+1}$ and $\alpha^{2n}(g)=gc_g^{-2n}$ for all $g\in H$ and $n\geq 1$. Since $G=G_{2'}$ has odd order, $m=|M|=|G\wedge G|$ is odd, and so $\alpha^{m}(g)=g^{-1}$ describes an isomorphism of $H$. This is only possible if $H$ is abelian, that is, if $G$ is cyclic, see Section \ref{secSM}. Back to the general case $G=G_2\times G_{2'}$, the same argument shows that $G_{2'}$ must be cyclic, hence it remains to consider the case $G=G_2$ in the following.  Since \[[h,g]=\alpha([g,h])=[\alpha(g),\alpha(h)]=[g^{-1}c_g,h^{-1}c_h]=[g^{-1},h^{-1}]=[g,h]^{h^{-1}g^{-1}}=[g,h]\] for all $g,h\in H$, we must have that $H'=M$ has exponent $2$. Thus, $\alpha$ is the identity on $M$, and so $\alpha^2(h)=\alpha(h^{-1}c_h)=hc_{h^{-1}}c_h=h$ for all $h\in H$ proves that  $\alpha$ has order 2. Note also that $[g,h]=c_{gh}c_gc_h$.  The map $\gamma\colon H\times H\to H'$, $(g,h)\mapsto [g,h]$, is a 2-cocycle in $Z^2(H,H')$. Since $H'$ is central, $\gamma$ induces a 2-cocycle $\delta\in Z^2(G,H')$. Since $G$ is abelian,  an isomorphism $G\wedge G\to H'$ is given by $g\wedge h\to [g',h']$, where $g',h'\in H$ are lifts of $g,h\in G$. This shows that the induced 2-cocycle $\delta$ lies in $Z^2(G,G\wedge G)$ and $\delta(g,h)=g\wedge h$ for all $g,h\in G$. Recall that if $h\in H$ and $z\in H'$, then $\alpha(h)=h^{-1}c_h$ and $(hz)^{-1}c_{hz}=\alpha(hz)=\alpha(h)\alpha(z)=h^{-1}c_hz$, which shows that $c_{hz}=c_h$. Thus for $g\in G$ we can define $\kappa(g)=c_{g'}$ where $g'\in H$ is a lift of $g$. This shows that $\delta(g,h)=\kappa(gh)\kappa(g)\kappa(h)$ with $\kappa(1)=1$, that is, $\delta$ is a 2-coboundary in $B^2(G,G\wedge G)$.

Conversely, let $G=G_2\times G_{2'}$ be abelian with cyclic $G_{2'}$ and  $G\wedge G$ of exponent $2$ such that  $\delta(g,h)=g\wedge h$ defines a 2-coboundary in $B^2(G_2,G_2\wedge G_2)$; by what is said above, it is sufficient to consider $G=G_2$. Since $\delta$ defines a $2$-coboundary, we have $g\wedge h=\delta(g,h)=\kappa(gh)\kappa(g)\kappa(h)$ for some map $\kappa\colon G\to G\wedge G$ with $\kappa(1)=1$. Let $H$ be a Schur cover of $G$ with natural projection $\pi\colon H\to G$, such that $M=\ker\pi$ satisfies  $M=H'\leq Z(H)$. Note that under the isomorphism $H'\to G\wedge G$, $[h,k]\mapsto \pi(h)\wedge \pi(k)$ we have $[h,k]=\delta(\pi(h),\pi(k))=\kappa(\pi(hk))\kappa(\pi(h))\kappa(\pi(k))$. Now define $\alpha\in\Aut(H)$ by $\alpha(h)=h^{-1}c_h$ where $c_h=\kappa(\pi(h))$; note that
\[\alpha(hk)=k^{-1}h^{-1}c_{hk}=h^{-1}k^{-1}[k^{-1},h^{-1}]c_{hk}=h^{-1}k^{-1}[k,h]c_{hk}=h^{-1}c_hk^{-1}c_k=\alpha(h)\alpha(k),\]
so $\alpha$ is indeed a homomorphism. Clearly, $\alpha$ acts as inversion (that is, as identity) on $M$, and as inversion on $H/M$. This proves the claim.
\end{proof}

\begin{proposition}\label{prop:ElAb2}
  If $G$ is an abelian $2$-group such that $\exp(G\wedge G)$ divides $2$, then $G\cong C_2^n \times C_{2^m}$ for some $n,m$ and $\tau(G)\cong\tilde K(G,3)$.
\end{proposition}
\begin{proof}
It is straightforward to see that an abelian $G$ as in the statement must be isomorphic to $C_2^n\times C_{2^m}$ for some $n$ and $m$. If $m\in\{0,1\}$, then $G$ is elementary abelian and Corollary \ref{corExp2} proves the claim. Now let $m\geq 2$ and let $\{x_1,\ldots,x_n,y\}\subseteq G$ be a generating set with $y$ of order $2^m$ and each $x_i$ of order $2$. A Schur cover of $G$ is $H=\langle x_1',\ldots,x_n',y',M\rangle$ where \[M=\langle z_{i,j},z_k : 1\leq i<j\leq n, 1\leq k\leq n\rangle\cong M(G)\cong G\wedge G\] is 2-elementary abelian and central in $H$, subject to the relations $(x_i')^2=(y')^{2^m}=1$, $[x_i',x_j']=z_{i,j}$ and $[x_k',y']=z_k$. Now inversion of $G$ lifts to an AI-automorphism of $H$ that fixes $M$ element-wise, and therefore $\tau(g)\cong \tilde K(G,3)$ by Theorem \ref{thm:epi2}.
\end{proof}

\begin{example}\label{exNonI}
 For $A=C_4^3$, we determine $\tau(A)\not\cong \tilde K(A,3)$ by computing with GAP that $\Aut(\tau(A))$ and $\Aut(\tilde K(G,3))$ have orders $94575592174780416$ and $283726776524341248$, respectively.   Since $A\wedge A$ has exponent 4, this is also an example showing that the assumptions in Proposition \ref{prop:ElAb2} cannot be relaxed. Similarly, it shows that Proposition  \ref{prop:rank2} cannot be extended to higher rank. A comparison of the automorphism group orders also shows that $\tau(B)\not\cong \tilde K(B,3)$ for $B=C_5^3$. 
\end{example}

\begin{proposition}\label{prop:rank2}
Let $G$ be an abelian group.
\begin{ithm}
\item Suppose all Sylow $p$-subgroups of $G$ have rank at most 2. Then $\tau(G)\cong \tilde K(G,3)$ if and only if the Sylow 3-subgroup of $G$ is cyclic.
\item If the Sylow 3-subgroup of $G$ has rank at least 2, then $\tau(G)\not\cong\tilde{K}(G,3)$.
\end{ithm}
\end{proposition}

\begin{proof}
Let $G=\prod_p G_p$ be the decomposition of $G$ into its Sylow subgroups. By \cite[Proposition~4.1]{Lie08} and \cite[Corollary 3.7]{Roc91}, we can also decompose $\tau(G)=\prod_p \tau(G_p)$ and $\tilde K(G,n)=\prod_p \tilde K(G_p,n)$, and every isomorphism $\tau(G)\to\tilde K(G,n)$ induces an isomorphism from $\tau(G_p)$ to $\tilde K(G_p,n)$ for every $p$. (We note that \cite[Corollary 3.7]{Roc91} considers $\nu(G)$, but it implies the fact needed for $\tau(G)$.) Thus it is sufficient to assume that $G$ is an abelian $p$-group.
\begin{iprf}
\item We have $G\cong C_m\times C_n$ with $m=p^a$ and $n=p^b$ for $a\geq b$. Recall from Section \ref{secSM} that $M(G)\cong G\wedge G\cong C_n$. Let $g$ and $h$ be generators of $C_m$ and $C_n$, respectively.  Considering the description of $\tau(G)$ as in Proposition \ref{prop:tauab}, set $g_1=(g,1;1)$, $h_1=(h,1;1)$, $g_2=(1,g;1)$, $h_2=(1,h;1)$, and $k=(1,1;h\wedge g)$. These elements form a polycyclic generating sequence of $\tau(G)$, with corresponding polycyclic presentation\[\tau (G)={\rm pc}\langle g_1,h_1,g_2,h_2,k\mid g_1^{m},g_2^{m},h_1^{n},h_2^{n},k^{n}, g_2^{h_1}=g_2k^{-1}, h_2^{g_1}=h_2k\rangle.\] Using the identification of $\tilde K(G,3)=G^2.(G\wedge G)$ as in Proposition \ref{prop:ab3}, we obtain
  \begin{eqnarray*}\tilde{K}(G,3)&=&{\rm pc}\langle g_1,h_1,g_2,h_2,k\mid g_1^{m},g_2^{m},h_1^{n},h_2^{n},k^{n}, h_1^{g_1}=h_1k^2,\\
&&\hspace*{3.3cm} g_2^{h_1}=g_2k^{-1}, h_2^{g_1}=h_2k, h_2^{g_2}=h_2k^2\rangle.
  \end{eqnarray*}
If $p\ne 3$, then, by von Dyck's Theorem,  $(g_1,h_1,g_2,h_2,k)\mapsto (g_1g_2^2,h_1,g_2g_1^2,h_2,k)$ extends to an isomorphism $\tilde{K}(G,3)\to\tau(G)$. If $p=3$ and $G$ has rank 2, then  $\tau(G)\not\cong\tilde{K}(G,3)$, see part b). If $G$ is a cyclic $3$-group, then $M(G)=1$, hence $\tau(G)=\tilde K(G,3)$ by Theorem \ref{thm:epi2}.

\item Let $G$ be an abelian $3$-group.  As $G$ is not cyclic, hence $Z^\wedge(G)\ne G$, it follows  that there exists $u\in G\setminus Z^\wedge(G)$ with $u^3\in Z^\wedge(G)$. Now Propositions~\ref{prop:ab3} and \ref{prop:tauab} imply $\tau(G)\not\cong \tilde K(G,3)$.
\end{iprf} 
 \end{proof}

Proposition \ref{prop:rank2}a) together with Proposition \ref{propabc} shows that there are infinitely many abelian groups $G$ such that $\tau(G)\cong \tilde K(G,3)$, but no Schur cover of $G$ has an AI-automorphism whose restriction to the Schur multiplier is inversion.

\section{Bogomolov multiplier}
\label{sec:bog}

Let $G$ be a group with AI-automorphism $\alpha$, and let $ {\Phi_\alpha}\colon\tau(G)\to K(G,3)$ be the epimorphism in Section \ref{secEpi}. Set \[M^\flat(G)=\langle [x,\ic{y}]: x,y\in G, [x,y]=1\rangle_{\tau(G)}\]and note that $M^\flat(G)$ is contained in the kernel of the commutator map $\kappa \colon[G,\ic{G}]_{\tau(G)}\to G'$. Define \[\tau^\flat(G)=\tau (G)/M^\flat (G).\] If $x$ and $y$ commute in $G$, then $ {\Phi_\alpha}([x,\ic{y}])=(x^{-1}x,y^{-1}y,\alpha([x,y]))=(1,1,1)$, therefore $ {\Phi_\alpha}$ induces an epimorphism $ \Phi_\alpha^\flat \colon \tau^\flat(G)\to K(G,3)$. Theorem \ref{thm:kernel} implies that the kernel of this map is $(\ker\kappa)/M^\flat (G)$, which is isomorphic to the {\em Bogomolov multiplier} $B_0(G)$ of $G$, see \cite{Mor12}.

\begin{corollary}\label{cor:seqB0}
The existence of an AI-automorphism of $G$ yields  a central extension
\[\begin{tikzcd}
    1\arrow{r} & B_0(G) \arrow{r} & \tau^\flat (G) \arrow{r} & K(G,3) \arrow{r} & 1.
\end{tikzcd}\]
\end{corollary}

\begin{proposition}
  \label{lem:KtildeB0}
Let $H$ be a Schur cover of a group $G$ with $H/M=G$. If $\alpha$ is an AI-automorphism of $H$, then  $\tilde{K}(G,3)\cong \tau^\flat(H)/\im\iota$ for the  monomorphism  $\iota \colon M^2\to \tau^\flat (H)$ given by
  \[(m_1,m_2)\mapsto m_1\ic{m_2}\prod\nolimits_{i=1}^\ell [\alpha ^{-1}(k_i),\icb{\alpha^{-1}(h_i)}]_{\tau^\flat(H)},\]
where the elements $h_i,k_i\in G$ are defined by $\alpha(m_1m_2)=m_2^{-1}m_1^{-1}[h_\ell,k_\ell]\ldots [h_1,k_1]$. 
\end{proposition}

\begin{proof}Since $M$ is abelian,  $M^2\cong K(M,3)$ with  isomorphism $(m_1,m_2)\mapsto (m_1,m_2,m_1^{-1}m_2^{-1})$. Note that $K(M,3)$ is naturally embedded in $K(H,3)$. From \cite[Proposition 6.12]{Mal99} we conclude that $B_0(H)$ is trivial, therefore $\Phi_\alpha^\flat\colon\tau^\flat (H)\to K(H,3)$ is an isomorphism by Corollary \ref{cor:seqB0}. Note that $\iota$ is the map that makes the following diagram commutative; in particular, $\iota$ is an injective homomomorphism, and the results follows from  taking quotients in diagram:
    \[\begin{tikzcd}
        M^2\arrow{r}{\cong} \arrow{d}{\iota} & K(M,3) \arrow{d}\\
        \tau^\flat(H) \arrow{r}{ \Phi_\alpha^\flat} & K(H,3).
      \end{tikzcd}\]
    \end{proof}

  \vspace*{-1ex}

 
\section{Computations}
\label{sec:comput}

\noindent If $G$ is a finite polycyclic group, then also $\tilde K(G,3)$ is polycyclic, see \cite[Proposition 1.5]{Lie08}. In this situation, the algorithms described in \cite{EN} can be used to compute $\tau(G)$; these algorithms are implemented in the software package Polycyclic, distributed with the computer algebra system GAP \cite{gap}.  Our explicit formulas in Section \ref{sec:explicit} can be used to compute a polycyclic presentation for $\tilde K(G,3)$. We have done this to test whether $\tau(G)$ and $\tilde K(G,3)$ are isomorphic for certain examples of groups (abelian, Frobenius, extra-special, \ldots). Even though there exist powerful algorithms for working with polycyclic groups, approaching this isomorphism problem with conventional methods poses a serious computational challenge. This is due to the fact that if $G$ is an abelian group of order $p^n$, then $\tilde K(G,3)$ and $\tau(G)$ are both large central extensions of $G\wedge G$ by $G^2$; they have class 2, order $p^{2n}|G\wedge G|$, and often seem indistinguishable. The latter is not a surprise, given the folklore conjecture that \emph{most}  $p$-groups have class 2: for example, note that among the 49499125314 groups of order at most $1024$ (up to isomorphism), 99.976\% of these are 2-groups and 98.595\% are 2-groups of class 2, see \cite[Section 4]{gnu}. A computational isomorphism test for these groups reduces to orbit calculations of huge matrix groups on very large vector spaces; often these computations turn out to be infeasible. For example, the powerful implementations of the $p$-group algorithms for automorphism groups and isomorphisms (provided by the GAP package Anupq) struggle to compute automorphisms and isomorphisms for $\tau(G)$ and $\tilde K(G,3)$ already for moderately sized $p$-groups such as $G=C_7^3$. Most of our computer experiments have therefore focused on groups of cubefree order, that is, groups whose order is not divisible by any prime power $p^3$. 

\begin{example}\label{ex:cf} 
In Table \ref{tabTK} we report on some example computations: there are 237 cubefree groups of order at most 100. Of these, 113 groups are abelian, 123 groups are non-abelian solvable, and 1 group is simple. Every abelian $G$ admits AI-automorphisms and, being cubefree, $\tau(G)\cong \tilde K(G,3)$ if and only if $G$ has a cyclic Sylow $3$-subgroup, see Proposition \ref{prop:rank2}. Our computations show that, with two exceptions,  $\tau(G)\cong\tilde K(G,3)$ if and only if $G$ has  AI-automorphisms. The exceptions are $A=C_3\times{\rm Alt}_4$ and $B=C_3^2\times D_{10}$; we have $Z(\tilde K(A,3))=C_6\times C_3$ and $Z(\tau(A))=C_6$, and $Z(\tilde K(B,3))= C_3^3$ and $Z(\tau(B))=C_3$.
\end{example}

\begin{table}[h] 
\caption{Statistics for solvable non-abelian groups of cubefree order at most 100}
\begin{tabular}{c|c|c}\label{tabTK}
   $\tau\cong \tilde K$ & has AI & \# groups\\\hline
    yes & yes & 96\\ 
    yes & no &  0\\
    no & yes &  2 \\
    no & no &   25\\[-1ex]   
   \end{tabular}
\end{table}

\begin{example}\label{exAI}
  Running over GAP's group database, there are  6505 non-abelian solvable groups of order $<256$; of these groups, 6127 have  AI-automorphisms. Note that every simple and every abelian group admits AI-automorphisms. This computation suggests that for many groups we can apply Corollary \ref{cor:seq} to describe $\tau(G)$ as a central extension of $H_2(G,\Z)$ by $K(G,3)$. Table \ref{tabTK}  suggests that the existence of AI-automorphisms for $G$ is strongly connected to the property $\tau(G)\cong\tilde K(G,3)$; cf.\ also Proposition \ref{prop:es}b,c).
\end{example}

\end{document}